\documentclass[11pt]{article}
\usepackage[T1]{fontenc}
\usepackage{graphicx}
\usepackage{color}
\usepackage{amsmath,amssymb}
\usepackage{latexsym, amsthm}
\usepackage{enumerate}
\usepackage{bbm}
\usepackage{hyperref}
\usepackage{authblk}
\usepackage{setspace} 
\usepackage{accents}
\large\normalsize

\theoremstyle{plain}
\newtheorem{theorem}{Theorem}[section]

\newtheorem{lemma}[theorem]{Lemma}

\theoremstyle{definition}
\newtheorem{remark}[theorem]{Remark}

\newcommand{\fmf}{Lip_{FM}(X)}

\numberwithin{equation}{section}
\newcommand{\n}{\mathbb{N}}
  
\renewcommand {\>}{\right\rangle}  
\newcommand{\norma}[1]{\left\|#1\right\|}

\title{
\bf{A Useful Version of the Central Limit Theorem for a~General Class of Markov Chains}
}
\author[1]{Dawid Czapla}
\author[1]{Katarzyna Horbacz}
\author[1]{Hanna Wojew\'odka}
\affil[1]{\small Institute of Mathematics, University of Silesia in Katowice, Bankowa 14, 40-007 Katowice, Poland}
\date{}

\begin{document}
\maketitle

\begin{abstract}
In the paper we propose certain conditions, relatively easy to verify, which ensure the central limit theorem for some general class of Markov chains. To justify the usefulness of our criterion, we further verify it for a~particular discrete-time Markov dynamical system. From the application point of view, the examined system
provides a useful tool in analysing the stochastic dynamics of gene expression in prokaryotes.
\end{abstract}




\section*{Introduction}\label{sec:intro}
The central limit theorem (CLT) is, beside the law of large numbers, the most fundamental limit theorem for random processes. It refers to the convergence in distribution of the standardized sample average towards the normal distribution.  
Although limit theorems for positive Harris recurrent Markov chains are already well-investigated (see \cite{mt}), they are still the subject of research for a~wider class of Markov chains. 
An interesting version of the CLT for stationary ergodic Markov chains was provided by M. Maxwell and M. Woodroofe in \cite{mw}. Recently, the CLT has also been established for various non-stationary Markov processes, e.g. the processes a~priori possessing unique stationary distributions, but not necessarily starting from them (see \cite{ghsz, klo, kw}).

Here, we also establish a version of the CLT for a quite general class of Markov chains. Our aim, however, is to provide a tool which may prove to be useful in studying, in terms of limit theorems, certain stochastic models developed in natural sciences (especially, molecular bilology). Therefore we do not require any form of continuous dependence of the given Markov chain on the initial conditions (as is necessary to assume for the results in \cite{ghsz,kw} to hold). We do not even directly require the exponential mixing property (see e.g. \cite{hairer} for the precise formulation). 
Instead, we propose a set of relatively easily verifiable conditions, which yield both the exponential ergodicity in the context of weak convergence of measures (according to  \cite[Theorem 2.1]{ks}), as well as the CLT. 

The class of Markov chains for which we establish the CLT may be shortly specified by the existence of an appropriate Markovian coupling whose transition function can be decomposed into two parts, one of which is contractive and dominant in some sense. The construction of such a~coupling is adapted from \cite{dawid2, ks}, which, in turn, is inspired by the prominent results of M. Hairer \cite{hairer}. Within this framework, we provide an exponential estimate of the mean distance between two coupled copies of the examined chain. This result, stated in Lemma \ref{cor:g_useful}, slightly generalizes the exponential mixing property obtained by R.~Kapica and M.~\'Sl\k{e}czka while proving \hbox{\cite[Theorem 2.1]{ks}}. The precise proof of this lemma is interesting itself, as well as it also clarifies the reasoning presented in \cite{ks}. In fact, Lemma \ref{cor:g_useful} plays a key role in this paper, since it allows us to carry out a~brief proof of the CLT (Theorem~\ref{thm:CTG}). 
The proof also appeals to the results of M. Maxwell and M.~Woodroof \cite{mw}, which make it more concise and less technical than the classical proofs, based directly on martingale methods. The proofs in \cite{hhsw} and \cite{horbacz_clt} are carried out in the same spirit, although only for some specific cases. 
It is also worth mentioning here that conditions proposed in this paper (namely hypothesis formulated in Sections \ref{sec:SGP} and \ref{sec:CLT}) yield the Donsker invariance principle for the CLT (cf. \cite{bill}), provided that the Markov chain is stationary (which in this case means that its initial distribution is exactly its unique invariant distribution, whose existence follows from the assumptions).

To justify the usefulness of the given criterion, we further apply it to a~particular discrete-time Markov dynamical system (cf.  \cite{asia,dawid}), for which the conditions proposed in \cite{ghsz} cannot be directly verified. The examined system has interesting biological interpretations. First of all, it can be viewed as the chain given by the post-jump locations of some piecewise-deterministic Markov process, which occurs in a simple model of  gene expression (cf. \cite{dawid,mtky}). More pecisely, it describes the amounts of protein molecues synthesized from a bacterial gene in random bursts. On the other hand, a special case of the above-mentioned model provides a mathematical framework for modelling the spatial distribution of the compounds involved in the gene autoregulation, i.e. the produced protein and its phosphorylated and dimerised form (for details, see \cite{hhs}). The latter also indicates the importance of considering a non-locally compact space as the state space in the abstract framework. 


The paper is organised as follows. In Section \ref{sec:1} we introduce notation and definitions used throughout the paper. First of all, we relate to the theory of Markov chains, discussed more widely e.g. in \cite{mt,revuz}. Further, we present some general facts concerning measures on a~path space, and introduce the notion of Markovian coupling. At the end of Section \ref{sec:1}, we recall two results on the properties of hitting and absorption times. In Section \ref{sec:SGP} we quote the criterion on the exponential ergodicity in the context of weak convergence of probability measures, established in \cite[Theorem 2.1]{ks}. Moreover, we also provide its stronger version, namely Lemma \ref{cor:g_useful}. The proof of the CLT (Theorem \ref{thm:CTG}) is given in Section \ref{sec:CLT}. Finally, in Section \ref{sec:ex} we establish the CLT for the Markov chain given by the post-jump locations of some piecewise deterministic Markov process considered in \cite{dawid}.

\section{Prelimenaries}\label{sec:1}

Within this section we establish notation and give some basic definitions used throughout the paper. 
We also recall some well-known facts, useful for further proofs. 

\subsection{Markov Operators}

First of all, let $\mathbb{R}$ denote the set of real numbers, $\mathbb{R}_+=[0,\infty)$, $\mathbb{N}=\{1,2,\ldots\}$ and $\mathbb{N}_0=\mathbb{N}\cup\{0\}$. 
We consider a complete separable metric space $(X,\varrho)$, i.e. a Polish space. By $\mathcal{B}_X$ we denote the $\sigma$-field of all Borel subsets of $X$. For any set $A$ and any point $x$, we use the symbols $\mathbbm{l}_A$ and $\delta_x$ to denote the indicator function of $A$ and the Dirac measure at $x$, respectively. Let us write $B_b(X)$ for the space of all bounded Borel measurable functions $f:X\to\mathbb{R}$, endowed with the supremum norm $\|f\|_{\infty}=\sup_{x\in X}|f(x)|$. Further, let $C_b(X)$ and $Lip_b(X)$ denote the subspaces of $B_b(X)$  consisting of all continuous and all Lipschitz-continuous functions, respectively. 
At some point we shall also refer to the space $\bar{B}_b(X)$ of functions $f:X\to\mathbb{R}$ which are Borel measurable and bounded below. Such a space contains, in particular, the so-called Lyapunov functions, which play an important role in our further analysis. By a \emph{Lyapunov function} we mean a map $V:X\to[0,\infty)$ which is continuous, bounded on bounded sets, and, in the case of unbounded $X$, satisfies $\lim_{\varrho(x,\bar{x})\to\infty}V(x)=\infty$ for some fixed point $\bar{x}\in X$. Moreover, for simplicity, we use the notation $\langle f,\mu\rangle:=\int_Xf(x)\,\mu(dx)$ for any $f\in\bar{B}_b(X)$ and any signed measure ($\sigma$-additive set function) $\mu:\mathcal{B}_X\to\mathbb{R}$.

The cone of all non-negative Borel measures on $X$ is denoted by $\mathcal{M}(X)$. We distinguish the following subsets of $\mathcal{M}(X)$:
\begin{align*}
&\mathcal{M}_{fin}(X)=\left\{\mu\in \mathcal{M}(X):\;\mu(X)<\infty\right\},\qquad\mathcal{M}_1(X)=\left\{\mu\in \mathcal{M}(X):\; \mu(X)=1\right\},\\
&\mathcal{M}_{1,s}^V(X)=\left\{\mu\in \mathcal{M}_1(X):\; \left\langle V^s,\mu\right\rangle<\infty\right\}
\end{align*}
for some $s>0$ and some Lyapunov function $V:X\to[0,\infty)$. 
For any $\mu\in \mathcal{M}_{fin}(X)$, we write supp$\,\mu$ for the support of $\mu$, i.e. $\text{supp}\,\mu=\{x\in X:\;\mu\left(B(x,r)\right)>0\;\text{for any}\; r>0\}$, where $B(x,r)=\{y\in X:\;\varrho(x,y)< r\}$. 
The set $\mathcal{M}_1(X)$ will be considered with the topology induced by the so-called Fortet-Mourier 
distance (see e.g. \cite{l_frac}), defined as follows:
$$d_{FM}(\mu_1,\mu_2)=\sup\left\{\left|\left\langle f,\mu_1-\mu_2\right\rangle\right|:\;f\in \fmf\right\}\quad\text{for}\quad \mu_1,\mu_2\in \mathcal{M}_1(X),$$
where $$\fmf=\{f\in Lip_b(X):\;\norma{f}_{BL}\leq 1\},\;\;\;\;\;\;\norma{f}_{BL}=\max(|f|_{Lip},\,\norma{f}_{\infty}),$$ 
and $|f|_{Lip}$ stands for the minimal Lipschitz constant of $f$.
Since $(X,\varrho)$ is assumed to be Polish, the convergence in $d_{FM}$ is equivalent to the weak convergence of probability measures. This assumption additionally implies completeness of the space $\left(\mathcal{M}_1(X),d_{FM}\right)$ (for proofs, see e.g. \cite{dudley}). 

A mapping $\Pi:X\times \mathcal{B}_X\to [0,1]$ is called a \emph{(sub)stochastic kernel} (or a transition \emph{(sub)probability function}) if $\Pi(\cdot,A):X\to[0,1]$ is a Borel measurable map for any fixed $A\in \mathcal{B}_X$, and $\Pi(x,\cdot):\mathcal{B}_X\to[0,1]$ is a (sub)probability Borel measure for any fixed $x\in X$. Every such kernel naturally induces two operators: $P:\mathcal{M}_{fin}(X)\to \mathcal{M}_{fin}(X)$ and \hbox{$U:{B}_b(X)\to {B}_b(X)$} given by
\begin{align}\label{def:markov_op}
&P\mu(A)=\int_X\Pi(x,A)\,\mu(dx)\;\;\text {for} \;\;\mu\in \mathcal{M}_{fin}(X),\; A\in \mathcal{B}_X,\\
\label{def:dual_op}
&Uf(x)=\int_Xf(y)\,\Pi(x,dy) \;\;\text{for}\;\; f\in {B}_b(X),\;x\in X.
\end{align}
If the kernel $\Pi$ is stochastic, then $P$ given by (\ref{def:markov_op}) is called a \emph{regular Markov operator}, whilst $U$ given by (\ref{def:dual_op}) is said to be its \emph{dual operator}. The duality relationship takes the form
\begin{align*}
\langle f,P\mu \rangle=\langle Uf,\mu\rangle &\quad\text{for}\quad f\in{B}_b(X),\; \mu\in \mathcal{M}_{fin}(X).
\end{align*}
Moreover, a regular Markov operator $P$ is said to be \emph{Feller} if $Uf\in C_b(X)$ for every $f\in C_b(X)$. Let us indicate that the above-defined mappings are related with one another in the following way \hbox{(cf. \cite{l_frac,z}):}
\begin{align*}
\Pi(x,A)=P\delta_x(A)=U\mathbbm{l}_A(x)&\quad\text{for}\quad x\in X,\; A\in \mathcal{B}_X.
\end{align*}
Further, note that any dual operator $U$ can be extended, in the usual way, to a linear operator on the space $\bar{B}_b(X)$. Hence, in particular, we are allowed to write $UV$, whenever $V$ is a~Lyapunov function. We should keep in mind, however, that it can happen that $UV(x)=\infty$ for some $x\in X$, as long as no additional assumptions are imposed. 

For any (sub)stochastic kernel $\Pi$, we can define the $n$-th step kernels $\Pi^n$, $n\in\mathbb{N}_0$, by setting
\begin{align*}
\begin{aligned}
&\Pi^0(x,A)=\delta_x(A),\quad \Pi^1(x,A)=\Pi(x,A)
,\\
&\Pi^n(x,A)=\int_X\Pi(y,A)\,\Pi^{n-1}(x,dy)
\quad \text{for}\quad 
x\in X,\; A\in\mathcal{B}_X,\;n\in\mathbb{N}.
\end{aligned}
\end{align*}

Let $P$ be an arbitrary regular Markov operator. 
We call $\mu_*\in \mathcal{M}_{fin}(X)$ an~\emph{invariant measure} of $P$ if $P\mu_*=\mu_*$. The operator $P$ is said to be \emph{exponentially ergodic in} $d_{FM}$ whenever it has a~unique invariant measure $\mu_*\in \mathcal{M}_1(X)$ and there exists $q\in (0,1)$ such that
\begin{align*}
d_{FM}(P^n\mu,\mu_*)\leq q^n C(\mu)\;\;\;\mbox{for any}\;\;\;\mu\in\mathcal{M}_{1,1}^V(X),\;n\in\n,
\end{align*}
where $C(\mu)$ is a constant which depends only on the initial distribution $\mu$.

Suppose that $(\phi_n)_{n\in\n_0}$ is a time-homogeneous $X$-valued Markov chain, defined on a~probability space $(\Omega, \mathcal{A},\mathbb{P})$. We say that $(\phi_n)_{n\in\n_0}$ has the one-step transition law determined by a stochastic kernel $\Pi$, if
\begin{align} \label{trans_n} 
\Pi(x,A)=\mathbb{P}(\phi_{n+1}\in A|\phi_n=x)\quad\text{for}\quad x\in X,\;A\in \mathcal{B}_X,\;n\in\mathbb{N}_0.
\end{align}

\subsection{Measures on a Path Space}\label{sec:measures}
Every stochastic kernel $\Pi^n$ determines the 
probabilities $\mathbb{P}_{x}^n$ of the form
\begin{align}\label{rule_1_dim}
\mathbb{P}_x^n(\cdot)=\Pi^n(x,\cdot)\quad\text{for}\quad x\in X,\;n\in\mathbb{N}_0.
\end{align}  
We further introduce the higher-dimensional distributions $\mathbb{P}^{1,\ldots,n}_x$ on $X^n$, for $x\in X$, defined (inductively on $n\in\mathbb{N}$) as follows:  
provided that the probability measures $\mathbb{P}^{1,\ldots,k}_x$ on $X^{k}$ have already been defined for every  $k<n$, the distribution $\mathbb{P}^{1,\ldots,n}_x$ is given as the unique measure which satisfies
\begin{align}\label{rule_n_dim}
\mathbb{P}^{1,\ldots,n}_x(A\times B)=\int_{A}\mathbb{P}^1_{z_{n-1}}(B)\,\mathbb{P}^{1,\ldots,n-1}_x(dz_1\times \ldots\times dz_{n-1}),\quad A\in \mathcal{B}_{X^{n-1}},\; B\in \mathcal{B}_X.
\end{align}

Let us now recall the following theorem (cf. \cite{mt,revuz}):

\begin{theorem}\label{thm:P_mu}
Consider $\Omega:=X^{\mathbb{N}_0}$ with the product topology, and let $(\phi_n)_{n\in\mathbb{N}_0}$ denote the sequence of mappings from $\Omega$ to $X$ given by $\phi_n(\omega)=x_n$ for $\omega=(x_0,x_1,\ldots)\in \Omega$. Then, for any $\mu\in \mathcal{M}_{1}(X)$, and any stochastic kernel $\Pi:X\times \mathcal{B}_X\to\left[0,1\right]$, there exists a probability measure $\mathbb{P}_{\mu}\in\mathcal{M}_1(\Omega)$, such that, for every $n\in\mathbb{N}$,
\begin{align}\label{eq:P_mu}
\mathbb{P}_{\mu}(A_0\times\ldots\times A_n\times X \times X \ldots)=\int_{A_0}\mathbb{P}^{1,\ldots,n}_x(A_1\times\ldots\times A_n)\,\mu(dx),\;\;A_0,\ldots,A_n\in\mathcal{B}_X,
\end{align}
where  $\mathbb{P}_x^{1,\ldots,n}$ are defined by \eqref{rule_1_dim}, \eqref{rule_n_dim}. In particular, $(\phi_n)_{n\in \mathbb{N}_0}$ is then a time-homogeneous Markov chain on the probability space $(\Omega,\mathcal{B}_{\Omega},\mathbb{P}_{\mu})$ with transition probability function $\Pi$ and initial distribution $\mu$. Clearly, 
$\mathbb{P}_{\mu}(B)$ is then the probabiliy of the event $\left\{(\phi_n)_{n\in\mathbb{N}_0}\in B\right\}$ for $B\in\mathcal{B}_{\Omega}$. 
\end{theorem}
{The Markov chain defined accordingly to Theorem \ref{thm:P_mu} will be further called a canonical Markov chain.} By convention, we will write 
$\mathbb{P}_x(B)=\mathbb{P}_{\mu}(B|\phi_0=x)$ for $B\in\mathcal{B}_{X^{\mathbb{N}_0}}$, and we will denote the expected values corresponding to $\mathbb{P}_x, \mathbb{P}_{\mu}\in \mathcal{M}_1(X^{\mathbb{N}_0})$ by $\mathbb{E}_{x}$, $\mathbb{E}_{\mu}$, respectively.

\subsection{A Coupling Method}\label{sec:coupling}

Let us now introduce a piece of notation related to a coupling method, often used to evaluate the rate of convergence to stationary distributions (more elaboration on coupling techniques can be found in \cite{hairer, hms,lind}). 

A~time-homogeneus Markov chain $(\phi^{(1)}_n,\phi^{(2)}_n)_{n\in\mathbb{N}_0}$ evolving on $X^2$ (endowed with the product topology) is said to be a \emph{Markovian coupling} of some stochastic kernel $\Pi$ whenever its transition law $C:X^2\times  \mathcal{B}_{X^2}\to\left[0,1\right]$ satisfies
$$
C(x,y,A\times X)= \Pi(x,A)\quad\text{and}\quad  C(x,y,X\times A)= \Pi(y,A)\quad\text{for any}\quad x,y\in X,\;A\in \mathcal{B}_X.
$$
By convention, the kernel $C$ itself is often called a coupling of $\Pi$, too.

It will be crucial in the analysis which follows that, for any transition probability function $\Pi$ and any substochastic kernel $Q:X^2\times\mathcal{B}_{X^2}\to [0,1]$ satisfying 
\begin{align}\label{def:substoch}
Q(x,y,A\times X)\leq \Pi(x,A)\quad\text{and}\quad
Q(x,y,X\times A)\leq\Pi(y,A)
\quad\text{for}\quad x,y\in X,\;A\in\mathcal{B}_X,
\end{align}
there exists a substochastic kernel $R:X^2\times \mathcal{B}_{X^2}\to[0,1]$ such that $C=Q+R$ is a~Markovian coupling of $\Pi$. Indeed, note that we can define $R$ by setting
\begin{align*}
R(x,y,A\times B)=\frac{\left(\Pi(x,A)-Q(x,y,A\times X)\right)\left(\Pi(y,B)-Q(x,y,X\times B)\right)}{\left(1-Q\left(x,y,X^2\right)\right)}\quad\text{for}\quad A,B\in\mathcal{B}_X,
\end{align*}
when $Q(x,y,X^2)<1$, and $R(x,y,A\times B)=0$ otherwise. Applying rules \hbox{(\ref{rule_1_dim})-(\ref{rule_n_dim})} to \linebreak{$C:X^2\times\mathcal{B}_{X^2}\to[0,1]$} and appealing to Theorem~\ref{thm:P_mu}, we can consider the canonical Markov chain $(\phi^{(1)}_n,\phi^{(2)}_n)_{n\in\mathbb{N}_0}$ with transition law $C$ and an arbitrarily fixed initial distribution $\alpha\in\mathcal{M}_1(X^2)$. We assume that the chain is defined on $((X^2)^{\mathbb{N}_0},\mathcal{B}_{(X^2)^{\mathbb{N}_0}},\mathbb{C}_{\alpha})$, where $\mathbb{C}_{\alpha}\in \mathcal{M}_1((X^2)^{\mathbb{N}_0})$ satisfies the appriopriate condition corresponding to (\ref{eq:P_mu}). 

Let us now consider the augmented space $\widehat{X^2}=X^2\times\{0,1\}$, splitted into $\widehat{X^2}_Q=X^2\times\{1\}$ and $\widehat{X^2}_R=X^2\times\{0\}$, as well as the stochastic kernel $\widehat{C}:\widehat{X^2}\times\mathcal{B}_{\widehat{X^2}}\to[0,1]$ determined by 
\begin{align}\label{def:C_hat}
\widehat{C}\left(x,y,\theta,A\times\{1\}\right)=Q(x,y,A)\quad\text{and}\quad 
\widehat{C}\left(x,y,\theta,A\times \{0\}\right)=R(x,y,A)
\end{align}
for $(x,y,\theta)\in \widehat{X^2}$ and $A\in\mathcal{B}_{X^2}$. 
{Let $\widehat{\alpha}\in\mathcal{M}_1(\widehat{X^2})$ be such that 
$\widehat{\alpha}(A\times\{0,1\})=\alpha (A)$ for any $A\in\mathcal{B}_{X^2}$. 
Then Theorem \ref{thm:P_mu} guarantees the existence of the canonical Markov chain $(\phi^{(1)}_n, \phi^{(2)}_n, \theta_n)_{n\in\mathbb{N}_0}$ with transition law $\widehat{C}$ and initial distribution $\widehat{\alpha}$, which is further referred to as an \emph{augmented coupling}. The chain is defined on $((\widehat{X^2})^{\mathbb{N}_0},\mathcal{B}_{(\widehat{X^2})^{\mathbb{N}_0}},\widehat{\mathbb{C}}_{\widehat{\alpha}})$, where $\widehat{\mathbb{C}}_{\widehat{\alpha}}$ is an appropriate probability measure on  $\mathcal{B}_{(\widehat{X^2})^{\mathbb{N}_0}}$.} 
Note that 
\begin{align}\label{coupling_marginals}
\widehat{\mathbb{C}}_{x,y,\theta}\left(\left(\phi^{(1)}_n,\phi^{(2)}_n\right)\in A\right)=\mathbb{C}^n_{x,y}(A)\;\;\; \text{for any}\;\;\; A\in\mathcal{B}_{X^2}\;\;\;\text{and any}\;\;\;(x,y,\theta)\in \widehat{X^2}.
\end{align}
By convention, we will further write $\mathbb{C}_{x,y}=\mathbb{C}_{\alpha}(\cdot|(\phi^{(1)}_0,\phi^{(2)}_0)=(x,y))$ for any $(x,y)\in X^2$ and $\widehat{\mathbb{C}}_{x,y,\theta}=\widehat{\mathbb{C}}_{\widehat{\alpha}}(\cdot|(\phi^{(1)}_0,\phi^{(2)}_0,\theta_0)=(x,y,\theta))$ for every $(x,y,\theta)\in\widehat{X^2}$. The expected values corresponding to the measures $\mathbb{C}_{x,y}\in \mathcal{M}_1((X^2)^{\mathbb{N}_0})$ and $\widehat{\mathbb{C}}_{x,y,\theta}\in \mathcal{M}_1((\widehat{X^2})^{\mathbb{N}_0})$ are denoted by $\mathbb{E}_{x,y}$ and $\widehat{\mathbb{E}}_{x,y,\theta}$, respectively.

In the research literature one can find the notion of a coupling time $\tau_{\text{couple}}$, which is the random moment at which both copies of the Markov chain, governed by $\Pi$, reach the same state for the first time, i.e. $\tau_{\text{couple}}=\min\{n\in\mathbb{N}_0: \phi^{(1)}_n=\phi^{(2)}_n\}$. If $\tau_{\text{couple}}<\infty$, then the so-called successful coupling, making chains stay together all the time, can be constructed (see \cite{lind}, \cite{mt}). Otherwise, one has to apply an asymptotic coupling, which is the case here. M. Hairer is the one who proposed (in \cite{hairer}) an effective method to couple asymptotically the whole trajectories of Markov chains which cannot actually meet at some finite time (cf. the example discussed in \cite{hms}, concerning some Markov chains on infinite dimensional spaces whose kernels, when starting from different initial conditions, induce mutually singular measures). Such an approach is extensively applied in this paper, especially in the proof of Lemma~\ref{cor:g_useful}.

\subsection{Auxiliary Results on Hitting and Absorption Times}

For any $A\in\mathcal{B}_X$, we define 
\begin{align}\label{def:tau}
\rho_A
=\inf\{n\in\mathbb{N}:\;\phi_n\in A\}\quad\text{and}\quad
\tau_A
=\inf\{n\in\mathbb{N}:\;\phi_k\in A\;\text{for all}\;{k\geq n}\},
\end{align}
which describe the first hitting time on $A$ and the time of absorption by $A$, respectively. At some point we will also make use of the variables 
\begin{align}\label{def:rho^m}
\rho_A^m
&=\inf\{n\geq m:\;\phi_n\in A\},\quad m\in\mathbb{N}.
\end{align}
Let us now quote two useful results proven in \cite{ks}.

\begin{lemma}[\text{\cite[Lemma 2.1]{ks}}]\label{lem:ks21}
Let $(\phi_n)_{n\in\mathbb{N}_0}$ be a time-homogeneous Markov chain evolving on $X$.  Assume that, for some bounded set $K\in \mathcal{B}_X$, there exist $\Lambda\in(0,1)$ and $c_{\Lambda}\in\mathbb{R}$ such that 
\begin{align}\label{eq:lem21i}
{\mathbb{E}}_{x}\left(\Lambda^{-\rho_K}\right)\leq c_{\Lambda}(1+V(x))\quad\text{for}\quad x\in X,
\end{align}
where $V:X\to[0,\infty)$ is a~Lyapunov function. 
Moreover, suppose that, for some $B\in\mathcal{B}_X$, there exist $\varkappa\in(0,1)$ and $c_{\varkappa}\in\mathbb{R}$ such that
\begin{align}\label{eq:lem21ii}
\mathbb{E}_x\left(\mathbbm{1}_{\{\rho_{X\backslash B}<\infty\}}\varkappa^{\rho_{X\backslash B}}\right)\leq c_{\varkappa}
\end{align}
and
\begin{align}\label{eq:lem21iii}
{
\inf_{x\in K}{\mathbb{P}}_{x}\left(\left\{\phi_n\in B\;\;\text{for all}\;\;n\in\mathbb{N}\right\}\right)>0.
}
\end{align}
Then there exist constants $\zeta\in(0,1)$ and $c_{\zeta}\in\mathbb{R}$ such that
{$$
{
\mathbb{E}_x\left(\zeta^{-\tau_B}\right)\leq c_{\zeta}(1+V(x))\quad\text{for}\quad x\in X.}
$$}
\end{lemma}

\begin{lemma}[\text{\cite[Lemma 2.2]{ks}}]\label{lem:ks22}
Assume that $(\phi_n)_{n\in\mathbb{N}_0}$ is a time-homogeneous Markov chain, evolving on $X$, with tranisition law $\Pi$ and the corresponding dual operator $U$, defined~by~\eqref{def:dual_op}. Further, suppose that there exist a~Lyapunov function $V:X\to[0,\infty)$ and constants $a\in(0,1)$, $b\in(0,\infty)$ such that 
$$UV(x)\leq aV(x)+b\;\;\;\mbox{for all}\;\;\;x\in X.$$
Then, for  
$$J=\left\{x\in X:\; V(x)<\frac{2b}{1-a}\right\},$$
there exist $\lambda\in(0,1)$ and $c_{\lambda}\in\mathbb{R}$ such that 
\begin{align}\label{eq:lem22}
{\mathbb{E}_{x}\left(\lambda^{-\rho_J^m}\right)\leq \lambda^{-m}c_{\lambda}(1+V(x))}\quad\text{for}\quad x\in X,\; m\in\mathbb{N}.
\end{align}
\end{lemma}

Although Lemma \ref{lem:ks22} is given in a slightly stronger version than  \cite[Lemma 2.2]{ks}, the proof is almost the same, so we will not repeat it.

\section{The Criterion on the Exponential Ergodicity}\label{sec:SGP}
This part of the paper draws heavily on ideas used e.g. in \cite{hairer, hhsw, ks, sleczka, hw}. The key result here is Lemma \ref{cor:g_useful}, which slightly strengthens the exponential mixing property (in $d_{FM}$) obtained and used in the proof of \hbox{\cite[Theorem~2.1]{ks}}. This lemma will be an essential tool in proving the CLT given in \hbox{Section~\ref{sec:CLT}}.

Let $(X,\varrho)$ be a Polish space, and suppose that we are given a transition probability function $\Pi:X\times\mathcal{B}_X\to[0,1]$. Assuming that $P:\mathcal{M}_{fin}(X)\to\mathcal{M}_{fin}(X)$ and \hbox{$U:\bar{B}_b(X)\to \bar{B}_b(X)$} denote the operators defined by \eqref{def:markov_op} and \eqref{def:dual_op}, respectively, we will use the following hypotheses:

\begin{itemize}
\item[(B0)] \phantomsection \label{cnd:B0}
The Markov operator $P$ has the Feller property.
\item[(B1)] \phantomsection \label{cnd:B1}
There exist a Lyapunov function $V:X\to[0,\infty)$ and constants $a\in (0,1)$ and $b\in(0,\infty)$ such that
$$UV(x)\leq aV(x)+b\quad\text{for every}\quad x\in X.$$ 
\end{itemize}
Moreover, we will require that there is a substochastic kernel $Q:X^2\times\mathcal{B}_{X^2}\to[0,1]$, satisfying~ (\ref{def:substoch}), which enjoys the following properties:
\begin{itemize}
\item[(B2)] \phantomsection \label{cnd:B2}
There exist $F\subset X^2$ and $\delta\in(0,1)$ such that 
\begin{align*}
\text{supp}\,Q(x,y,\cdot)\subset F\quad\text{and}\quad
\int_{X^2}\varrho(u,v)\,Q(x,y,du\times dv)\leq \delta \varrho(x,y)\quad\text{for}\quad(x,y)\in F.
\end{align*}
\item[(B3)] \phantomsection \label{cnd:B3}
Letting $U(r)=\{(u,v)\in F:\varrho(u,v)\leq r\}$, $r>0$, we have 
$$\inf_{(x,y)\in F}Q\big(x,y,U\left(\delta\varrho(x,y)\right)\big)>0.$$
\item[(B4)] \phantomsection \label{cnd:B4}
There exist constants $\beta\in (0,1]$ and $c_{\beta}>0$ such that
$$Q\left(x,y,X^2\right)\geq 1-c_{\beta}\varrho^{\beta}(x,y)\quad\text{for every}\quad (x,y)\in F.$$
\item[(B5)] \phantomsection \label{cnd:B5}
There exists 
a coupling $(\phi^{(1)}_n,\phi^{(2)}_n)_{n\in \mathbb{N}_0}$ of $\Pi$ with transition law $C\geq Q$ (cf.~Section~\ref{sec:coupling}) such that for some $\Gamma>0$ and
\begin{align}\label{def:K}
K:=\left\{(x,y)\in{X^2}:\;(x,y)\in F\;\text{and}\; V(x)+V(y)<\Gamma\right\}
\end{align}   
we can choose $\gamma\in(0,1)$ and $c_{\gamma}>0$ for which
$$\mathbb{E}_{x,y}(\gamma^{-\rho_K})\leq c_{\gamma},\quad\text{whenever}\quad V(x)+V(y)<4b(1-a)^{-1},$$
where, in this case, 
\begin{align}\label{def:tau}
\rho_K
=\inf\left\{n\in\mathbb{N}:\;(\phi^{(1)}_n,\phi^{(2)}_n)\in K\right\}.
\end{align}
\end{itemize}

The theorem we quote below is proven by R. Kapica and M. \'Sl\k{e}czka in \cite{ks}.

\begin{theorem}\label{thm:spectral_gap}
Suppose that $\Pi:X\times\mathcal{B}_X\to[0,1]$  is a transition probability function such that conditions \hyperref[cnd:B0]{(B0)}-\hyperref[cnd:B5]{(B5)} hold with some substochastic kernel $Q:X^2\times\mathcal{B}_{X^2}\to[0,1]$ satisfying~\eqref{def:substoch}. Then the Markov operator $P$, defined by \eqref{def:markov_op}, possesses a unique invariant measure $\mu_*\in \mathcal{M}_{1}(X)$ such that $\mu_*\in\mathcal{M}_{1,1}^V(X)$, where $V$ is determined by \hyperref[cnd:B1]{(B1)}. Moreover, there exist constants $q\in(0,1)$ and $c>0$ such that
$$d_{FM}(P^n\mu,\mu_*)\leq q^n c\left(1+\langle V,\mu\rangle+\langle V,\mu_*\rangle\right)\quad \text{for every}\quad \mu\in \mathcal{M}_{1,1}^V(X),\;n\in\mathbb{N}_0.$$
\end{theorem}

As mentioned earlier, apart from Theorem \ref{thm:spectral_gap}, we will also need an intermediate result (although in a slightly stronger version than the one given in \cite{ks}). More precisely, for a suitably constructed Markovian coupling $(\phi^{(1)}_n,\phi^{(2)}_n)_{n\in\mathbb{N}_0}$ of $\Pi$, we will provide an estimation of the expression $\mathbb{E}_{x,y}|g(\phi_n^{(1)})-g(\phi_n^{(2)})|$, where $x,y\in X$ and $g\in Lip_b(X)$. Aiming to formulate and prove this result, we shall first make a few technical observations (cf.~\cite{dawid2}). 

Assume that $\Pi$ and $Q$ stand for the kernels considered in Theorem \ref{thm:spectral_gap}. Let $R$ be a~substochastic kernel on $X^2\times\mathcal{B}_{X^2}$ such that $C=Q+R$ is a coupling of $\Pi$ such that \hyperref[cnd:B5]{(B5)} holds, and let $\widehat{C}$ denote the augmented coupling of $\Pi$, i.e. the stochastic kernel on $\widehat{X^2}\times \mathcal{B}_{\widehat{X^2}}$, determined by (\ref{def:C_hat}) (cf. Section \ref{sec:coupling}). The Markov chains governed by the kernels $C$ and $\widehat{C}$ will be denoted by $(\phi^{(1)}_n,\phi^{(2)}_n)_{n\in\mathbb{N}_0}$ and $(\phi^{(1)}_n,\phi^{(2)}_n,\theta_n)_{n\in\mathbb{N}_0}$, repsectively.

Let us define
\begin{align}\label{def:Gamma0}
\Gamma_0=\sup\{\varrho(x,y):(x,y)\in K\},
\end{align}
where $K$ is given by (\ref{def:K}). Due to the definition of $K$ and the fact that $(x,y)\mapsto V(x)+V(y)$ is a~Lyapunov function, we see that $\Gamma_0<\infty$. 

Now, let $(x,y,\theta)\in \widehat{X^2}$ and fix arbitrary $n, M,N\in\mathbb{N}$ such that $n>M>N$. Further, consider the random times
$$\rho_K^{N}=\inf\left\{n\geq N:\left(\phi^{(1)}_n,\phi^{(2)}_n\right)\in K\right\},\;\;\;\rho_K=\rho_K^1,$$
and
$$\tau:=\tau_{\widehat{X^2}_Q}=\inf\left\{n\in\mathbb{N}:\left(\phi^{(1)}_k,\phi^{(2)}_k,\theta_k\right)\in \widehat{X^2}_Q \text{ for all }k\geq n\right\}$$
(cf. (\ref{def:tau}) and (\ref{def:rho^m})). 
Moreover, let us also introduce 
\begin{align}\label{def:HNn}
\mathcal{H}_{N,n}=\bigcap_{j=N}^n\left\{\theta_j=1\right\}\quad\text{and}\quad \mathcal{H}_{N,n}^c=(\widehat{X^2})^{\mathbb{N}_0}\backslash \mathcal{H}_{N,n}.
\end{align}
Obviously $$\widehat{\mathbb{C}}_{x,y,\theta}\left(\mathcal{H}_{N,n}^c\right)=\widehat{\mathbb{C}}_{x,y,\theta}\left(\bigcup_{j=N}^n\{\theta_j=0\}\right)\leq \widehat{\mathbb{C}}_{x,y,\theta}(\tau>N).$$
Using the notation $\widehat{\mathbb{C}}_{x,y,\theta}|_E:=\widehat{\mathbb{C}}_{x,y,\theta}(\cdot \cap E)$ for $E\in (\widehat{X^2})^{\mathbb{N}_0}$, we can write
\begin{align*}
\widehat{\mathbb{C}}_{x,y,\theta}\leq \widehat{\mathbb{C}}_{x,y,\theta}|_{\left\{\rho_K^{N}\leq M\right\}\cap \mathcal{H}_{N,n}}+\widehat{\mathbb{C}}_{x,y,\theta}|_{\left\{\rho_K^{N}>M\right\}}+\widehat{\mathbb{C}}_{x,y,\theta}|_{\mathcal{H}_{N,n}^c}.
\end{align*}
which, in view of (\ref{coupling_marginals}), gives
\begin{align}\label{3parts}
\mathbb{E}_{x,y}|f(\phi_n^{(1)})-f(\phi_n^{(2)})|&=\int_{X^2}|f(u)-f(v)|\,{\mathbb{C}}^n_{x,y}(du\times dv)\nonumber\\ 
&\leq \int_{X^2}\varrho(u,v)\,\widehat{\mathbb{C}}_{x,y,\theta}|_{\left\{\rho_K^{N}\leq M\right\}\cap \mathcal{H}_{N,n}}\left(\left(\phi^{(1)}_n,\phi^{(2)}_n\right)\in du\times dv\right)\nonumber\\
&+2{\mathbb{C}}_{x,y}\left(\rho_K^{N}>M\right)
+2\widehat{\mathbb{C}}_{x,y,\theta}\left(\tau>N\right)\;\;\;\text{for}\;\;\;f\in \fmf.
\end{align}

{\begin{lemma}\label{lem:3parts}
Under the assumptions of Theorem \ref{thm:spectral_gap},  
there exist constants $c_1,c_2,c_3\geq 0$,  $q_1,q_2,q_3\in(0,1)$ and $p\geq 1$ such that, for any $(x,y,\theta)\in \widehat{X^2}$ and $n,N,M\in\mathbb{N}$ satisfying $n>N>M$, the following inequalities hold:
\begin{align}
\label{cnd:I1} &I_1:=\int_{X^2}\varrho(u,v)\,\widehat{\mathbb{C}}_{x,y,\theta}|_{\left\{\rho_K^{N}\leq M\right\}\cap \mathcal{H}_{N,n}}\left(\left(\phi^{(1)}_n,\phi^{(2)}_n\right)\in du\times dv\right)\leq c_1q_1^{n-M},\\[0.08cm]
\label{cnd:I2} &I_2:={\mathbb{C}}_{x,y}\left(\rho_K^{N}>M\right)\leq c_2q_2^{M-pN}(1+V(x)+V(y)),\\[0.12cm]
\label{cnd:I3} &I_3:=\widehat{\mathbb{C}}_{x,y,\theta}\left(\tau>N\right)\leq c_3q_3^N(1+V(x)+V(y)),
\end{align}
where $\mathcal{H}_{N,n}$ is given by (\ref{def:HNn}). 
\end{lemma}}

\begin{proof}
 
For $i\in\mathbb{N}$ and $(u,v)\in X^2$, we define $_K Q^i(u,v,\cdot):\mathcal{B}_{X^2}\to[0,1]$ by setting
\begin{align*}
&_K Q^i(u,v,A)\\
&=\widehat{\mathbb{C}}_{u,v,\theta}\left(\left(\phi^{(1)}_i,\phi^{(2)}_i,\theta_i\right)\in (K\cap A)\times\{1\},\;\left(\phi^{(1)}_k,\phi^{(2)}_k,\theta_k\right)\in \left(X^2\backslash K\right)\times\{1\}\;\text{for}\;k< i\right)\\
&=\widehat{\mathbb{C}}_{u,v,\theta}\left(\left\{\left(\phi^{(1)}_i,\phi^{(2)}_i\right)\in A\right\}\cap\left\{\rho_K=i\right\}\cap\left\{\theta_k=1\;\text{for}\; k\leq i\right\}\right)\quad \text{for any}\quad A\in\mathcal{B}_{X^2}.
\end{align*}
Obviously, $_KQ^i\left(u,v,X^2\backslash K\right)=0$. Now, let $n>N>M$. Due to the definitions of $\mathcal{H}_{N,n}$ and $\rho_K^{N}$,   
we obtain
\begin{align*}
&\widehat{\mathbb{C}}_{x,y,\theta}|_{\left\{\rho_K^N=i\right\}\cap \mathcal{H}_{N,n}}\left((\phi^{(1)}_n,\phi^{(2)}_n)\in A\right)\\
&\qquad=\int_{X^2}\int_K {Q^{n-i}}(w,z,A) \,_KQ^{i-N+1}(u,v,dw\times dz)\,\mathbb{C}^{N-1}_{x,y}(du\times dv)
\end{align*}
for $A\in X^2$, $(x,y,\theta)\in \widehat{X^2}$ and $i\in\{N,\ldots,n\}$. 
Clearly, $\{\rho_K^N\leq M\}=\bigcup_{i=N+1}^M\{\rho_K^N=i\}$, and hence, for $(x,y)\in X^2$, we have
\begin{align*}
\begin{aligned}
I_1\leq \sum_{i=N}^M\int_{X^2}\int_K\int_{X^2}\varrho(s,t)\,Q^{n-i}(w,z,ds\times dt)\,{_KQ}^{i-N+1}(u,v,dw\times dz)\,\mathbb{C}_{x,y}^{N-1}(du\times dv).
\end{aligned}
\end{align*}
From assumption \hyperref[cnd:B2]{(B2)} and the definition of $K$ it follows that
\begin{align*}
\begin{aligned}
I_1&\leq\sum_{i=N}^M\int_{X^2}\int_K\int_{F}\varrho(s,t)\,Q^{n-i}(w,z,ds\times dt)\,{_KQ}^{i-N+1}(u,v,dw\times dz)\,\mathbb{C}_{x,y}^{N-1}(du\times dv)\\
&\leq\sum_{i=N}^M\delta^{n-i}\int_{X^2}\int_K\varrho(w,z)\,{_KQ}^{i-N+1}(u,v,dw\times dz)\,\mathbb{C}_{x,y}^{N-1}(du\times dv)\\
&\leq \Gamma_0\delta^{n-M} \int_{X^2}\left(\sum_{i=N}^M{_KQ}^{i-N+1}(u,v,K)\right)\,\mathbb{C}_{x,y}^{N-1}(du\times dv)\\
&\leq \Gamma_0\delta^{n-M}\int_{X^2}\mathbb{C}_{u,v}\left(\rho_K^N\leq M-N+1\right)\,\mathbb{C}_{x,y}^{N-1}(du\times dv)\leq \Gamma_0\delta^{n-M}\quad\text{for}\quad (x,y)\in X^2,
\end{aligned}
\end{align*}
whence \eqref{cnd:I1} holds with $q_1=\delta$ and $c_1=\Gamma_0$, where $\Gamma_0$ is given by (\ref{def:Gamma0}).

Let us now introduce $\overline{V}(x,y):=V(x)+V(y)$ for $(x,y)\in X^2$ and
$$J=\left\{(x,y)\in {X^2}:\;\overline{V}(x,y)\leq {4b}{(1-a)^{-1}}\right\}.$$ 
Obviously, $\overline{V}$ is a~Lyapunov function on $X^2$, which, due to \hyperref[cnd:B1]{(B1)}, satisfies the inequality $U\overline{V}(x,y)\leq a\overline{V}(x,y)+2b\;\;\;\mbox{for any}\;\;\;(x,y)\in X^2.$ Consequently, referring to Lemma \ref{lem:ks22}, we can choose $\lambda\in(0,1)$ and $c_{\lambda}>0$ so that 
\begin{align}\label{estim:rho_J}
{\mathbb{E}}_{x,y}\left(\lambda^{-\rho_J^{m}}\right)\leq \lambda^{-m}c_{\lambda}(1+\overline{V}(x,y))\;\;\;\text{for all}\;\;\;(x,y)\in X^2,\; m\in\mathbb{N}.
\end{align}
Define \hbox{$T:(X^2)^{\mathbb{N}_0}\to(X^2)^{\mathbb{N}_0}$} by \hbox{$T\left((x_n,y_n)_{n\in\mathbb{N}_0}\right)=\left(x_{n+1},y_{n+1}\right)_{n\in\mathbb{N}_0}$}, and let
$$\mathcal{F}_{\rho_J^{N}}=\{A\in\mathcal{F}:\,\{\rho_J^{N}=k\}\cap A\in\mathcal{F}_k\,\text{ for }\,k\in\mathbb{N}_0\},$$ where $(\mathcal{F}_k)_{k\in\n_0}$ stands for the natural filtration of $(\phi^{(1)}_k,\phi^{(2)}_k)_{k\in\n_0}$. Now, put $\Lambda:=\max\{\lambda,\gamma\}$. Using the fact that $\rho_K^{N}\leq \rho_J^{N}+\rho_K\circ T^{\rho_J^{N}}$, and, further, applying sequentially the strong Markov property, condition \hyperref[cnd:B5]{(B5)} and inequality (\ref{estim:rho_J}) with $m=N$,  
we obtain
\begin{align}\label{eq:ad_lem21i}
\begin{aligned}
{\mathbb{E}}_{x,y}\left(\Lambda^{-\rho_K^{N}}\right)
&\leq{\mathbb{E}}_{x,y}\left(\Lambda^{-\rho_J^{N}}\Lambda^{-\rho_K\circ T^{\rho_J^{N}}}\right)
\leq{\mathbb{E}}_{x,y}\left(\lambda^{-\rho_J^{N}}{\mathbb{E}}_{x,y}\left(\gamma^{-\rho_K\circ T^{\rho_J^{N}}}|\mathcal{F}_{\rho_J^{N}}\right)\right)\\
&= {\mathbb{E}}_{x,y}\left(\lambda^{-\rho_J^{N}}{\mathbb{E}}_{\left(\phi^{(1)}_{\rho_J^{N}},\phi^{(2)}_{\rho_J^{N}}\right)}\left(\gamma^{-\rho_K}\right)\right)\\
&\leq c_{\gamma}\mathbb{E}_{x,y}\left(\lambda^{-\rho_J^{N}}\right)\leq \lambda^{-N}c_{\gamma}c_{\lambda}(1+\overline{V}(x,y)),
\end{aligned}
\end{align}
for some $c_{\gamma}>0$. Then, the Markov inequality yields that
\begin{align*}
\begin{aligned}
I_2=\mathbb{C}_{x,y}\left(\rho_K^{N}>M\right)&\leq \Lambda^{M}\lambda^{-N}c_{\gamma}c_{\lambda}(1+\overline{V}(x,y))
= \Lambda^{M-N \log_{\Lambda}\lambda}c_{\gamma}c_{\lambda}(1+\overline{V}(x,y))\\
&\leq \Lambda^{M-pN}c_{\gamma}c_{\lambda}(1+\overline{V}(x,y))\;\;\;\text{for all}\;\;\; (x,y)\in X^2,
\end{aligned}
\end{align*}
where $p=\left\lceil\log_{\Lambda}\lambda\right\rceil\geq 1.$
Hence, taking $c_2=c_{\gamma}c_{\lambda}$ and $q_2=\Lambda$, we obtain \eqref{cnd:I2}.

Finally, we have to deal with component $I_3$. For this purpose, we shall use Lemma \ref{lem:ks21} for the chain $(\phi^{(1)}_n,\phi^{(2)}_n,\theta_n)_{n\in\mathbb{N}_0}$ with $B=\widehat{X^2}_Q$ and $\widehat{K}:=K\times\{0,1\}$ in the role of $K$. Note that, due to (\ref{eq:ad_lem21i}), we have
\begin{align*}
{\widehat{\mathbb{E}}}_{x,y,\theta}\left(\Lambda^{-\widehat{\rho}_{\widehat{K}}}\right)= \mathbb{E}_{x,y}(\Lambda^{-\rho_K})\leq{\mathbb{E}}_{x,y}\left(\Lambda^{-\rho_K^{N}}\right)\leq \lambda^{-N}c_{\gamma}c_{\lambda}(1+\overline{V}(x,y)),\;\;\;(x,y,\theta)\in\widehat{X^2},
\end{align*} 
where $\widehat{\rho}_{\widehat{K}}$ is the first hitting time on $\widehat{K}$ for the augmented coupling of $P$. This corresponds to (\ref{eq:lem21i}) in Lemma \ref{lem:ks21}. Further, the Jensen inequality, together with assumption \hyperref[cnd:B2]{(B2)}, implies that, for any $(x,y)\in F$ and any $k\in\mathbb{N}$,
\begin{align}\label{estim:H2more}
\begin{aligned}
\int_{X^2}\varrho^{\beta}(u,v)\,Q^k(x,y,du\times dv)
&\leq\int_{X^2}\int_{X^2}\varrho^{\beta}(u,v)\,Q(w,z,du\times dv)\,Q^{k-1}(x,y,dw\times dz)\\
&\leq \int_{X^2}\left(\int_{X^2}\varrho(u,v)\,Q(w,z,du\times dv)\right)^{\beta}\,Q^{k-1}(x,y,dw\times dz)\\
&\leq\delta^{\beta}\int_{X^2}\varrho^{\beta}(w,z)\,Q^{k-1}(x,y,dw\times dz)\leq\ldots\leq \delta^{\beta k}\varrho^{\beta}(x,y).
\end{aligned}
\end{align}
Applying sequentially \hyperref[cnd:B4]{(B4)}, (\ref{estim:H2more}) and (\ref{def:Gamma0}) we conclude that, for every $(x,y,\theta)\in K\times\{0,1\}$,
\begin{align*}
\begin{aligned}
\widehat{\mathbb{C}}_{x,y,\theta}\left(\rho_{\widehat{X^2}_R}=k\right)&=\int_{X^2}\left(1-Q(u,v,X^2)\right)\,Q^{k-1}(x,y,du\times dv)\\
&\leq c_{\beta}\int_{X^2}\varrho^{\beta}(u,v)\,Q^{k-1}(x,y,du\times dv)\leq c_{\beta}\delta^{\beta(k-1)}\Gamma_0^{\beta},
\end{aligned}
\end{align*}
which gives (\ref{eq:lem21ii}) with $B=\widehat{X^2}_Q$, $\varkappa=\delta^{\beta}$ and $c_{\varkappa}=c_{\beta}(\Gamma_0/\delta)^{\beta}\sum_{k=1}^{\infty}\varkappa^{2k}<\infty$. 
Finally, we need to establish (\ref{eq:lem21iii}).
From \hyperref[cnd:B4]{(B4)} and (\ref{estim:H2more}) it follows that, for any $(x,y)\in F$ and any $k\in\mathbb{N}$,
\begin{align}\label{estim:H4more}
\begin{aligned}
Q^k(x,y,X^2)
&=\int_{X^2}Q(u,v,X^2)\,Q^{k-1}(x,y,du\times dv)\\
&\geq Q^{k-1}(x,y,X^2)-c_{\beta}\int_{X^2}\varrho^{\beta}(u,v)\,Q^{k-1}(x,y,du\times dv)\\
&\geq Q^{k-1}(x,y,X^2)-c_{\beta}\delta^{\beta (k-1)}\varrho^{\beta}(x,y)\\
&\geq\ldots\geq 1-c_{\beta}\sum_{i=0}^{k-1}\delta^{\beta i}\varrho^{\beta}(x,y)\geq 1-\frac{c_{\beta}}{1-\delta^{\beta}}\varrho^{\beta}(x,y).
\end{aligned}
\end{align}
Assumption \hyperref[cnd:B3]{(B3)} guarantees that $\varphi:=\inf_{(x,y)\in F}Q\big(x,y,U\left(\delta\varrho(x,y)\right)\big)>0$,
and we can further show that
\begin{align}\label{estim:H3more}
Q^k\left(x,y,U \left(\delta^k \varrho(x,y)\right)\right)\geq \varphi^k\quad \text{for all}\quad (x,y)\in F,\;k\in\n.
\end{align}
Indeed, we have
\begin{align*}
Q\big(u,v,U\left(\delta^{k+1}\varrho(x,y)\right)\big)\geq Q\big(u,v,U(\delta\varrho(u,v)\big)\geq \varphi\quad\text{for}\quad (u,v)\in U(\delta^k\varrho(x,y)),
\end{align*}
since $U(\delta\varrho(u,v))\subset U(\delta^{k+1}\varrho(x,y))$ for all $(u,v)\in U(\delta^k\varrho(x,y))$, and therefore
\begin{align*}
Q^{k+1}\left(x,y,U\left(\delta^{k+1}\varrho(x,y)\right)\right)
&\geq\int_{U\left(\delta^{k}\varrho(x,y)\right)}Q\left(u,v,U\left(\delta^{k+1}\varrho(x,y)\right)\right)\,Q^k(x,y,du\times dv)\\
&\geq\varphi \,Q^k\left(x,y,U\left(\delta^{k}\varrho(x,y)\right)\right)\geq\ldots\geq \varphi^{k+1}\quad\text{for}\quad (x,y)\in F.
\end{align*}
Applying sequentially (\ref{estim:H4more}), (\ref{estim:H2more}) and (\ref{estim:H3more}), we see that, for any $(x,y)\in K\subset F$ and $k,m\in\mathbb{N}$, the following inequalities hold:
\begin{align*}
Q^{k+m}(x,y,X^2)
&\geq\int_{U(\delta^{k}\varrho(x,y))} Q^m\left(u,v,X^2\right)\,Q^k(x,y,du\times dv)\\
&\geq \int_{U(\delta^{k}\varrho(x,y))}\left(1-\frac{c_{\beta}}{1-\delta^{\beta}}\varrho^{\beta}(u,v)\right) \,Q^k(x,y,du\times dv)\\
&\geq \left(1-\frac{c_{\beta}}{1-\delta^{\beta}}\delta^{\beta k}\varrho(x,y)^{\beta}\right) Q^k\big(x,y,U(\delta^k\varrho(x,y))\big)
\geq \left(1-\frac{c_{\beta}\delta^{\beta k}}{1-\delta^{\beta}}\Gamma_0^{\beta}\right) \varphi^k,
\end{align*}
where $\Gamma_0$ is determined by (\ref{def:Gamma0}). 
We can now choose $k_0\in\mathbb{N}$ so large that \hbox{$\delta^{\beta k_0}<\linebreak(1-\delta^{\beta})/\left(2c_{\beta}\Gamma_0^{\beta}\right)$,} which immediately implies that
\begin{align*}
Q^{k_0+m}\left(x,y,X^2\right)\geq \epsilon\quad\text{for all}\quad m\in\mathbb{N},\; (x,y)\in K,\quad\text{where}\quad\epsilon:=\varphi^{k_0}/2.
\end{align*}
Since the lower bound does not depend on $m\in\mathbb{N}$, we also have 
$$\widehat{\mathbb{C}}_{x,y,\theta}\left(\left\{ (\phi^{(1)}_k,\phi^{(2)}_k,\theta_k)\in\widehat{X^2}_Q\;\text{for all}\;k\in\mathbb{N}\right\}\right)=\lim_{m\to\infty}Q^{m}(x,y,X^2)\geq \epsilon$$
for any $(x,y,\theta)\in K\times\{0,1\}$,  
which completes the proof of (\ref{eq:lem21iii}) and therefore, due to Lemma \ref{lem:ks21}, implies the existence of constants $\zeta\in(0,1)$ and $c_{\zeta}>0$
such that
\begin{align*}
\widehat{\mathbb{E}}_{x,y,\theta}\left(\zeta^{-\tau}\right)\leq c_{\zeta}(1+\overline{V}(x,y))\quad\text{for}\quad (x,y,\theta)\in \widehat{X^2}.
\end{align*} 
Hence, using the Markov inequality, we obtain
\begin{align}\label{estim3}
I_3=\widehat{\mathbb{C}}_{x,y,\theta}\left(\tau>N\right)\leq c_{\zeta}\zeta^{N}(1+\overline{V}(x,y)),
\end{align}
and one can easily observe that \eqref{cnd:I3} holds with $c_3=c_{\zeta}$ and $q_3=\zeta$.
\end{proof}

We are now in a position to prove the announced lemma.

\begin{lemma}\label{cor:g_useful}
Under the assumptions of Theorem \ref{thm:spectral_gap},  
there exist $q\in(0,1)$ and $c>0$ such that
\begin{align}
\label{eq:lemma_thesis}
\mathbb{E}_{x,y}\left|g\left(\phi_n^{(1)}\right)-g\left(\phi_n^{(2)}\right)\right|\leq c\norma{g}_{BL}q^n(1+V(x)+V(y))
\end{align}
for all $(x,y)\in X^2$, $g\in Lip_b(X)$ and $n\in\mathbb{N}_0$.
\end{lemma}
\pagebreak
\begin{proof}
Let $g\in Lip_b(X)$. 
According to (\ref{3parts}) and Lemma~\ref{lem:3parts}, we know that there exist constants $c_1,c_2,c_3\geq 0$,  $q_1,q_2,q_3\in(0,1)$ and $p\geq 1$ such that the inequality
\begin{align*}
\mathbb{E}_{x,y}\left|g\left(\phi_n^{(1)}\right)-g\left(\phi_n^{(2)}\right)\right|&=\int_{X^2}|g(u)-g(v)|\,{\mathbb{C}}^n_{x,y}(du\times dv)\\
&\leq |g|_{Lip}\left(c_1q_1^{n-M}\right)+2\|g\|_{\infty}\left(c_2q_2^{M-pN}(1+V(x)+V(y))\right)\\
&\quad\,+2\|g\|_{\infty}\left(c_3q_3^N(1+V(x)+V(y))\right)
\end{align*}
holds for all $x,y\in X$ and $n,M,N\in\mathbb{N}$ satisfying $n>N>M$. 
Now, define $n_0=\lceil 4p\rceil$ and fix an~arbitrary $n>n_0$. Letting $N=\lfloor n/(4p)\rfloor$ and $M=\lceil n/2\rceil$, we obtain
\begin{align*}
 \mathbb{E}_{x,y}\left|g\left(\phi_n^{(1)}\right)-g\left(\phi_n^{(2)}\right)\right|\leq \norma{g}_{BL}\bar{c} q^n(1+V(x)+V(y))\;\;\;\mbox{for all}\;\;\;n>n_0,
\end{align*}
where 
$$q=\max\left\{q_1^{1/2},q_2^{1/4},q_3^{1/(4p)}\right\}\in (0,1),\quad \bar{c}=\max\left\{q_1^{-1},q_3^{-1}\right\}(c_1+2c_2+2c_3)>0.$$
Since $\sup_{u,v\in X}|g(u)-g(v)|\leq 2\norma{g}_{BL}$, we see that the assertion of Lemma \ref{cor:g_useful} holds for arbitrary $n\in\mathbb{N}_0$ and $c:=2\max(\bar{c},1)q^{-n_0}$. 
\end{proof}

\section{The CLT}\label{sec:CLT}
The essential ideas of this section are motivated by \cite{hhsw} and \cite{horbacz_clt}, both based on \cite{mw}. 

Let $(X,\varrho)$ be a Polish space, and let $(\phi_n)_{n\in\mathbb{N}_0}$ be an $X$-valued time-homogeneous Markov chain with transition law $\Pi$ and arbitrary initial distribution $\mu\in\mathcal{M}_1(X)$. As before, $P$ and $U$ will denote the operators defined by \eqref{def:markov_op} an \eqref{def:dual_op}, respectively.

Assume that $\mu_*\in\mathcal{M}_1(X)$ is a~unique invariant measure of $P$. For any $n\in\mathbb{N}$ and any Borel function $g:X\to\mathbb{R}$ we define 
\begin{align*}
s_n(g)=\frac{g(\phi_1)+\ldots+g(\phi_{n})}{\sqrt{n}},
\end{align*}
\begin{align}\label{def:sigma}
\sigma^2(g)=\lim_{n\to\infty}\mathbb{E}_{\mu_*}\left(s_n^2\left(g\right)\right),
\end{align}
and write 
$\Phi  s_n(g)$ for the distribution of $s_n(g)$. Moreover, we define $\bar{g}=g-\langle g,\mu_*\rangle$.

For any given Borel function $g:X\to\mathbb{R}$ such that $\langle g^2,\mu_*\rangle<\infty$, we say that the CLT holds for $\left(g(\phi_n)\right)_{n\in\mathbb{N}_0}$, if 
$\sigma^2(\bar{g})<\infty$ 
and 
$\Phi s_n(\bar{g})$ converges weakly to $\mathcal{N}(0,\sigma^2(\bar{g}))$, as $n\to\infty$.

Now, let $D[0,1]$ denote the Skorochod space, i.e. the collection of all c\'adl\'ag functions on $[0,1]$ (cf. \cite{bill}). For any Borel function $g:X\to\mathbb{R}$, we introduce a process $({B}_n(g))_{n\in\mathbb{N}}$ with values in $D[0,1]$ by setting
\begin{align*}
{B}_n(g)(t)=\frac{1}{\sqrt{n}}\left(g\left(\phi_1\right)+\ldots+g\left(\phi_{\lceil nt\rceil }\right)\right),\;\;\;0\leq t<1,\;\;\;
\text{and}\;\;\;{B}_n(g)(1)={B}_n(g)(1-)
\end{align*}
for every $n\in\mathbb{N}$, 
where $\lceil a\rceil$ is the ceiling function of $a\in\mathbb{R}$. 

For any given Borel function $g:X\to\mathbb{R}$ such that $\langle g^{r},\mu_*\rangle<\infty$ for some $r>2$, we say that $(g(\phi_n))_{n\in\mathbb{N}_0}$ satisfies \emph{the Donsker invariance principle for the CLT} (the functional CLT), if $\sigma^2(\bar{g})<\infty$ 
and ${B}_n(\bar{g})$ converges weakly to $\sigma(\bar{g})B$ in the space $D[0,1]$, as $n\to\infty$, where $B$ is a~standard Brownian motion on $\left[0,1\right]$.

\subsection{An Auxiliary Result}
An important step in the proof of our main result (given in Section \ref{subsec_CLT}) follows from \linebreak\hbox{\cite[Corollary~1]{mw}} and \cite[Corollary~4]{mw}. For the convenience of the reader, we summarize their key statements in the lemma below.

Let us note that, if $g:X\to\mathbb{R}$ is integrable with respect to $\mu\in\mathcal{M}_{fin}(X)$, then $Ug(x)$, given by \eqref{def:dual_op}, is well-defined and finite for $\mu$-almost all $x\in X$. Therefore, for every $n\in\mathbb{N}$ and any Borel function $g:X\to\mathbb{R}$, we can define
\begin{align*}
\mathcal{V}_n\,g=\sum_{k=1}^{n}U^kg.
\end{align*}

\begin{lemma}\label{thm:MW}
Let $(\phi_n^*)_{n\in\mathbb{N}_0}$ be an $X$-valued time-homogeneous Markov chain with transition law $\Pi$, which posseses a unique invariant distribution \hbox{$\mu_*\in \mathcal{M}_1(X)$}. Further, suppose that $g:X\to\mathbb{R}$ is a~Borel function satisfying $\langle g^2,\mu_*\rangle<\infty$, for which there exist $\alpha<1/2$, $n_0\in\n$ and $c>0$ such that
\begin{align}\label{cond_MW}
\left\langle (\mathcal{V}_n\,\bar{g})^2,\,\mu_*\right\rangle^{1/2}\leq c n^{\alpha}\;\;\;\mbox{for all}\;\;\;n\geq n_0.
\end{align}
Then, assuming that $\mu_*$ is the initial distribution of $(\phi_n^*)_{n\in\mathbb{N}_0}$, the following statements are fulfilled:
\begin{itemize}
\item[(a)] \phantomsection\label{cnd:a}
The CLT holds for $\left({g}(\phi_n^*)\right)_{n\in\mathbb{N}_0}$.
\item[(b)] \phantomsection\label{cnd:b} 
If, additionally, $\langle |g|^r,\mu_*\rangle<\infty$ for some $r>2$, then $\left({g}(\phi_n^*)\right)_{n\in\mathbb{N}_0}$ obeys the Donsker invariance principle for the CLT.
\end{itemize}
\end{lemma} 
Note that, while formulating this lemma, we have already taken into account the fact that, in the case of Polish spaces (which is the one that we consider here), the convergence in the Prokhorov metric (see e.g. \cite{bill} or \cite{mw} for the definition) is equivalent to the weak convergence of probability measures.

\subsection{The Main result}\label{subsec_CLT}

Before we formulate the main theorem of this section, we need to strengthen condition \hyperref[cnd:B1]{(B1)} to the following form:
\begin{itemize}
\item[(B1)$^\prime$]\label{cnd:B1p}\phantomsection
There exist a Lyapunov function $V:X\to[0,\infty)$ and constants $a\in (0,1)$ and \hbox{$b\in(0,\infty)$} such that
\[UV^2(x)\leq \left(aV(x)+b\right)^2\quad\text{for every}\quad x\in X.\]
\end{itemize}
Obviously, due to the H\"older inequality, hypothesis \hyperref[cnd:B1p]{(B1)$^{\prime}$} implies \hyperref[cnd:B1]{(B1)}. Indeed,
\begin{align*}
UV(x)=\left\langle V,P\delta_x\right\rangle\leq\left\langle V^2,P\delta_x\right\rangle^{1/2}=\left(UV^2(x)\right)^{1/2}\leq aV(x)+b\;\;\;\mbox{for any}\;\;\;x\in X.
\end{align*}

\begin{theorem}\label{thm:CTG}
Suppose that $\Pi:X\times\mathcal{B}_X\to[0,1]$  is a transition probability function such that conditions \hyperref[cnd:B0]{(B0)}-\hyperref[cnd:B5]{(B5)} with \hyperref[cnd:B1]{(B1)} strengthened to \hyperref[cnd:B1p]{(B1)$^{\prime}$} hold with some substochastic kernel \linebreak\hbox{$Q:X^2\times\mathcal{B}_{X^2}\to[0,1]$} satisfying~(\ref{def:substoch}). Further, let $(\phi_n)_{n\in\mathbb{N}_0}$ be an $X$-valued time-homogeneous Markov chain with transition law given by $\Pi$ and initial distribution $\mu\in\mathcal{M}_{1,1}^V(X)$, where $V$ is a~Lyapunov function appearing in \hyperref[cnd:B1p]{(B1)$^{\prime}$}. Then the CLT holds for $\left(g(\phi_n)\right)_{n\in\mathbb{N}_0}$ whenever $g\in Lip_b(X)$.  Moreover, if the initial distribution of $(\phi_n^*)_{n\in\mathbb{N}_0}$ is equal to the unique invariant measure of $P$, then $(g(\phi_n^*))_{n\in\mathbb{N}_0}$ obeys the Donsker invariance principle for the CLT.
\end{theorem}

\begin{proof}
The proof proceeds in three steps.
\newline
{\textbf{Step I. }} First of all, note that Theorem \ref{thm:spectral_gap} provides the existence of a unique invariant measure $\mu_*\in\mathcal{M}_1(X)$ for $P$. 
We need to show that $\mu_*\in \mathcal{M}_{1,2}^V(X)$. To do this, fix an arbitrary $x\in X$, and observe that condition \hyperref[cnd:B1p]{(B1)$^{\prime}$} yields
\begin{align}\label{estim}
\begin{aligned}
U^nV^2(x)
&\leq a^{2n}V^2(x)+2a^nbV(x)\sum_{i=0}^{n-1}a^i+b^2\sum_{i=0}^{n-1}a^{2i}+2ab^2\sum_{i=0}^{n-1}a^i\\
&\leq a^{2n}V^2(x)+a^n\frac{2b}{1-a}V(x)+\frac{b^2}{1-a^2}+\frac{2ab^2}{1-a},
\end{aligned}
\end{align}
which may be easily proven inductively on $n\in\mathbb{N}$. 
Further, for every $k\in\mathbb{N}$, define $\widetilde{V}_k:X\to[0,k]$ by $\widetilde{V}_k=\min(k,V^2)$.
Clearly $\widetilde{V}_k\in C_b(X)$ for each $k\in\mathbb{N}$. Hence, referring to Theorem~\ref{thm:spectral_gap}, we have 
\begin{align}\label{eq:thm5.1}
\langle \widetilde{V}_k,\mu_*\rangle=\lim_{n\to\infty}\langle \widetilde{V}_k,P^n\delta_x\rangle=\lim_{n\to\infty}U^n\widetilde{V}_k(x)\quad\text{for every}\quad k\in\mathbb{N}.
\end{align}
Now, observe that $(\widetilde{V}_k)_{k\in\mathbb{N}}$ is a non-decreasing sequence of non-negative functions satisfying \linebreak$\lim_{k\to\infty}\widetilde{V}_k(y)={V}^2(y)$ for all $y\in X$. Therefore, using the Monotone Convergence Theorem, together with (\ref{eq:thm5.1}) and (\ref{estim}), we obtain
\begin{align*}
\left\langle V^2,\mu_*\right\rangle
=\lim_{k\to\infty}\left\langle \widetilde{V}_k,\mu_*\right\rangle
=\lim_{k\to\infty}\lim_{n\to\infty}U^n\widetilde{V}_k(x)
\leq\limsup\limits_{n\to\infty} U^n V^2(x)
\leq b^2\frac{2a+2a^2+1}{1-a^2},
\end{align*}
which implies that, indeed, $\mu_*\in \mathcal{M}_{1,2}^V(X)$.\newline
{\textbf{Step II. } }  
Let $g\in Lip_b(X)$. Then, obviously, $\langle |g|^{r},\mu_*\rangle<\infty$ for any $r\geq 2$. In order to apply both parts \hyperref[cnd:a]{(a)} and \hyperref[cnd:b]{(b)} of Lemma \ref{thm:MW}, we only need to verify condition (\ref{cond_MW}). It is therefore enough to find some upper bound, independent of $n\in\mathbb{N}$, for the expression
\begin{align*}
\left\langle (\mathcal{V}_n\,\bar{g})^2,\mu_*\right\rangle
=\int_X\left(\sum_{k=1}^{n}U^k\bar{g}(x)\right)^2\mu_*(dx).
\end{align*}
Observe that Lemma \ref{cor:g_useful} immediately implies that there exist $q\in(0,1)$ and $c>0$ such that, for any $x,y\in X$,
\begin{align*}
|U^k g(x) - U^k g(y)|&=\left|\left\langle g, \Pi^n(x,\cdot)-\Pi^n(y,\cdot)\right\rangle\right|\\
&= \left|\int_{X^2}g(u)\,\mathbb{C}^n_{x,y}(du\times dv)-\int_{X^2}g(v)\,\mathbb{C}^n_{x,y}(du\times dv)\right|\\
&\leq \int_{X^2}|g(u)-g(v)|\,\mathbb{C}^n_{x,y}(du\times dv)\leq c\norma{g}_{BL}q^n(1+V(x)+V(y)).
\end{align*}
Hence, recalling that $\mu_*$ is invariant for $P$, we obtain
\begin{align*}
\left|\sum_{k=1}^{n} U^k\bar{g}(x)\right|
&\leq\sum_{k=1}^{n}\left|U^k g(x)-\<U^k g,\mu_* \>\right|
\leq\sum_{k=1}^{n}\int_X|U^k g(x)-U^k g(y)|\,\mu_*(dy)\\
&\leq \sum_{k=1}^{n}\int_X c\norma{g}_{BL}q^k(1+V(x)+V(y)) \mu_*(dy)
\leq \frac{qc\norma{g}_{BL}}{1-q}(1+V(x)+\langle V,\mu_*\rangle),
\end{align*}
for any $x\in X$ and any $n\in\mathbb{N}$. 
Now, since $\mu_*\in \mathcal{M}_{1,2}^V(X)$, which, due to the H\"older inequality, entails $\mu_*\in \mathcal{M}_{1,1}^V(X)$, we have
\begin{align}\label{eq:indep_of_n}
\begin{aligned}
\left\langle \left(\mathcal{V}_n\,\bar{g}\right)^2,\mu_*\right\rangle
&\leq \frac{q^2c^2\norma{g}_{BL}^2}{(1-q)^2}\int_X\left(1+\langle V,\mu_*\rangle+ V(x)\right)^2\,\mu_*(dx)\\
&\leq2 \frac{q^2c^2\norma{g}_{BL}^2}{(1-q)^2}\left(\Big(1+\langle V,\mu_*\rangle\Big)^2+\left\langle V^2,\mu_*\right\rangle\right)<\infty.
\end{aligned}
\end{align}
Since the above estimation is independent of $n\in\mathbb{N}$, condition (\ref{cond_MW}) is satisfied. Hence, Lemma \ref{thm:MW} implies both the CLT and its functional version for the Markov chain $(g(\phi_n^*))_{n\in\mathbb{N}_0}$, whenever $\mu_*$ is the distribution of $\phi_0$ (i.e the chain is stationary). \newline
\textbf{{Step III. } } Let us now establish the assertion of Theorem \ref{thm:CTG} for a non-stationary Markov chain. For the clarity of the arguments, let us indicate the initial distribution of $(\phi_n)_{n\in\mathbb{N}_0}$ in the upper index of $s_n(\bar{g})$, $n\in\mathbb{N}_0$.  
According to statement \hyperref[cnd:a]{(a)} of Lemma \ref{thm:MW} (whose hypothesis has already been verified in Step II), we know that $\Phi s_n^{\mu_*}(\bar{g})$ converges weakly to $\mathcal{N}(0,\sigma^2(\bar{g}))$, as $n\to\infty$. Therefore it is enough to prove that
\begin{align}\label{eq:enough1}
\lim_{n\to\infty}\left|\left\langle f, \Phi s_n^{\delta_x}(\bar{g})-\Phi s_n^{\delta_y}(\bar{g})\right\rangle\right|=0\;\;\;\text{for every}\;\;\; f\in \fmf\;\;\;\text{and any}\;\;\;x,y\in X.
\end{align}
Indeed, equality (\ref{eq:enough1}), together with the Dominated Convergence Theorem, implies that 
\begin{align}\label{eq:enough2}
\lim_{n\to\infty}\left|\left\langle f,\Phi s_n^{\mu}(\bar{g})\right\rangle -\left\langle f, \Phi s_n^{\mu_*}(\bar{g})\right\rangle\right|=0\quad \text{for any}\quad f\in \fmf.
\end{align}
Since $\Phi s_n^{\mu_*}(\bar{g})$ converges weakly to $\mathcal{N}(0,\sigma^2(\bar{g}))$, as $n\to\infty$, we in particular obtain
\[\lim_{n\to\infty}\left|\left\langle f,\Phi s_n^{\mu_*}\left(\bar{g}\right)\right\rangle-\left\langle f,\mathcal{N}\left(0,\sigma^2(\bar{g})\right)\right\rangle\right|=0\quad\text{for any}\quad f\in \fmf.\]
This, together with \eqref{eq:enough2}, gives
$$\lim_{n\to\infty}\left|\left\langle f, \Phi s_n^{\mu} (\bar{g})\right\rangle - \left\langle f, \mathcal{N}\left(0, \sigma^2(\bar{g})\right)\right\rangle\right| = 0\quad\text{for any}\quad f\in  \fmf,
$$
According to the fact that, in the case of Polish spaces, the convergence in $d_{FM}$ is equivalent to the weak convergence of probability measures, we obtain the desired conclusion.

It now remains to prove \eqref{eq:enough1}. Let $f\in \fmf$ be arbitrary. By virtue of Lemma \ref{cor:g_useful}, we have 
\begin{align*}
&\left|\left\langle f, \Phi s_n^{\delta_x}(\bar{g}) -\Phi s_n^{\delta_y}(\bar{g}) \right\rangle\right|=\Bigg|\int_{X^n}f\left(\frac{\bar{g}(u_1)+\ldots+\bar{g}(u_n)}{\sqrt{n}}\right)\,\mathbb{P}^{1,\ldots,n}_x(du_1\times\ldots\times du_n)\\
&\quad\,-\int_{X^n}f\left(\frac{\bar{g}(v_1)+\ldots+\bar{g}(v_n)}{\sqrt{n}}\right)\,\mathbb{P}^{1,\ldots,n}_y(dv_1\times\ldots\times dv_n)\Bigg|\\
&\leq\int_{X^{2n}}\left|\frac{g(u_1)+\ldots+g(u_n)}{\sqrt{n}}-\frac{g(v_1)+\ldots+g(v_n)}{\sqrt{n}}\right|\,
{\mathbb{C}}_{x,y}^{1,\ldots,n}\left(du_1\times dv_1\times\ldots\times du_n\times dv_n\right)\\
&\leq \frac{1}{\sqrt{n}}\sum_{i=1}^n\int_{X^2}|g(u)-g(v)|\,\mathbb{C}^i_{x,y}(du\times dv)
\leq \frac{c\norma{g}_{BL}}{\sqrt{n}(1-q)}(1+V(x)+V(y)),
\end{align*}
for some $q\in(0,1)$ and $c>0$. Hence, (\ref{eq:enough1}) and (\ref{eq:enough2}) are established and the proof of the CLT (for non-stationary Markov chains) is now completed. 
\end{proof}

\begin{remark}
Analyzing the proof of Theorem \ref{thm:CTG} shows that its assertion remains valid under two more general (and simultaneously, much more abstract) hypotheses, namely:
\begin{itemize}
\item[(i)] condition \hyperref[cnd:B1p]{(B1)$^{\prime}$} is fulfilled;
\item[(ii)] there exists a Markovian coupling $(\phi^{(1)}_n,\phi^{(2)}_n)_{n\in\mathbb{N}_0}$ of $\Pi$ for which condition \eqref{eq:lemma_thesis} is satisfied.
\end{itemize}
If we now compare \hyperref[cnd:B1p]{(B1)$^{\prime}$} and \eqref{eq:lemma_thesis} with the assumptions of \cite[Theorem 5.1]{ghsz}, we see that none of the results need not imply the other.
\end{remark}

\section{An Abstract Markov Model for Gene Expression}\label{sec:ex}

Within this section we indicate the usefulness of Theorem \ref{thm:CTG}. For this reason we refer to an~abstract model, which occurs mainly in gene expression analysis (cf. \cite{hhs,dawid,mtky}). Such a model has already been investigated in terms of 
its exponential ergodicity and the strong law of large numbers in \cite{asia, dawid}.

\subsection{The Structure and Assumptions of the Model}\label{sec:model}

Consider a separable Banach space $(H,\|\cdot\|)$ and a~closed subset $Y$ of $H$. By $B(h,r)$, where $h\in H$ and $r>0$, we will denote  the open ball in $H$ of radius $r$ centered at $h$. Further, assume that $(\Theta,\mathcal{B}(\Theta),\Delta)$ is a~topological measure space with a $\sigma$-finite Borel mesure $\Delta$. For simplicity, in the rest of the paper, we will write $d\theta$ instead of $\Delta(d\theta)$. Moreover, fix $N\in\n$, and let $I:=\{1,\ldots,N\}$ be endowed with the discrete metric $(i,j)\mapsto {d}(i,j)$, that is, ${d}(i,j)=1$ for $i\neq j$ and ${d}(i,j)=0$ for $i=j$. 

We are concerned with a random dynamical system $(Y(t))_{t\in\mathbb{R}_+}$ evolving through random jumps on the space $Y$.  It is required that the jumps occur at random moments $\tau_n$, \hbox{$n\in\mathbb{N}$}, coinciding with the jump times of a Poisson process with intensity $\lambda$. Between the jumps the system evolves deterministically. It is driven by a finite number of semiflows \hbox{$S_i:\mathbb{R}_+\times Y\to Y$}, $i\in I$, which are assumed to be continuous with respect to each variable. These semiflows are switched from jump to jump, according to a~matrix of continuous functions \hbox{$\pi_{ij}:Y\to\left[0,1\right]$}, $i,j\in I$, satisfying
$\sum_{j\in I}\pi_{ij}(y)=1$ for any $y\in Y,\,i\in I$. More formally, we have
\[Y(t)=S_{\xi_n}\left(t-\tau_n,Y\left(\tau_n\right)\right)\;\;\; \text{for}\;\;\; t\in[\tau_n,\tau_{n+1}),\]
where $\xi_n$ is an $I$-valued random variable  describing the choice of a semiflow directly after the $n$-th jump.

For $n\in\mathbb{N}$, the post-jump location $Y(\tau_n)$ is a result of a transformation of the state $Y(\tau_n-)$ just before the jump, determined by a function randomly selected among all possible ones \hbox{$w_{\theta}:Y\to Y$}, $\theta\in\Theta$, additionally perturbed by a random shift $H_n$ within an \hbox{$\varepsilon$-neighbourhood}. In other words, we have $Y(\tau_n)=w_{\theta_n}(Y(\tau_n-))+H_n$. 

It is required that all the maps $(y,\theta)\mapsto w_{\theta}(y)$ are continuous, and also that there exists  $\varepsilon^*>0$ for which 
\[w_{\theta}(y)+h\in Y\;\;\; \text{whenever} \;\;\; h\in B(0,\varepsilon^*),\;\;\; \theta \in \Theta,\;\;\;y\in Y.\] 
We further require that all the disturbances $H_n$ have a common disribution \hbox{$\nu^{\varepsilon}\in\mathcal{M}_1(H)$}, supported on the ball $B(0,\varepsilon)$, where $\varepsilon\in\left[0,\varepsilon^*\right]$ (in the case where $\varepsilon=0$, we set \hbox{$B(0,\varepsilon)=\{0\}$}).

Moreover, we will  assume that the probabilities of choosing $w_{\theta}$ (at the jump times) are determined by the place-dependent density functions $\theta\mapsto p(y,\theta)$ with $y\in Y$, where \hbox{$p: Y\times \Theta \to \left[0,\infty\right)$} is a continuous function such that $\int_{\Theta} p(y,\theta)\,d\theta=1$ for any $y\in Y$.  

In the analysis that follows, we will focus on the sequence of random variables $(Y_n)_{n\in\mathbb{N}_0}$ given by the post-jump locations of $(Y(t))_{t\in\mathbb{R}_+}$, that is, $Y_n=Y(\tau_n)$ for $n\in\mathbb{N}$. Such a~sequence can be defined on a~suitable probability space, say $(\Omega, \mathcal{F}, \mathbb{P})$, by
\begin{equation}\label{def:Y_n}
Y_{n+1}=w_{\theta_{n+1}}(S_{\xi_{n}}(\Delta \tau_{n+1},Y_{n}))+H_{n+1}\;\;\;\text{for}\;\;\;n\in\mathbb{N}_0,
\end{equation}
where the above variables and their distributions    
are specified by the following conditions:
\begin{itemize}
\item[(i)] $Y_0:\Omega\to Y$ and $\xi_0:\Omega\to I$ have arbitrary and fixed  distributions. 
\item[(ii)] $(\tau_n)_{n\in\mathbb{N}_0}$ is a strictly increasing sequence of random variables $\tau_n :\Omega\to \left[0,\infty\right)$, $n\in \mathbb{N}_0$, such that $\tau_0=0$ and $\tau_n\to\infty$, as $n\to\infty$. Moreover, the increments $\Delta \tau_{n+1}:=\tau_{n+1}-\tau_{n}$ are mutually independent and have the common exponential distribution with intensity $\lambda>0$.
\item[(iii)] $H_n:\Omega\to Y$, $n\in\mathbb{N}$, are identically distributed with $\nu^{\varepsilon}$.
\item[(iv)] $\theta_n :\Omega\to \Theta$ and $\xi_n :\Omega\to I$, $n\in \mathbb{N}$, are defined (inductively) in the following way:
\begin{align*}
&\mathbb{P}(\theta_{n+1}\in D\;|\;S_{\xi_{n}}(\Delta \tau_{n+1},Y_{n})=y;\,W_n)
=\int_{D} p(y,\theta)\,d\theta 
\;\;\;\text{for}\;\;\; D\in\mathcal{B}(\Theta),\; y\in Y,\;n\in\mathbb{N}_0,\\
&\mathbb{P}(\xi_{n+1}=j\;|\;Y_{n+1}=y,\, \xi_{n}=i;\,W_n)=\pi_{ij}(y) 
\;\;\;\text{for}\;\;\; y\in Y,\;i, j\in I,\;n\in\mathbb{N}_0,
\end{align*}
where $W_0=(Y_0,\;\xi_0)$ and $W_n=(W_0,\;H_1,\ldots,H_n,\;\tau_1,\ldots,\tau_n,\;\theta_1,\ldots,\theta_n,\;\xi_1,\ldots,\xi_n)$ for $n\in\mathbb{N}$.
\end{itemize}
Simultaneously, we require that, for any $n\in\mathbb{N}_0$, the variables $\Delta\tau_{n+1}$, $H_{n+1}$, $\theta_{n+1}$ and $\xi_{n+1}$ are ({mutually}) conditionally independent given $W_n$, and that $\Delta\tau_{n+1}$ and $H_{n+1}$ are independent of $W_n$.

Moreover, we impose the following assumptions, adapted from \cite{dawid}:
\begin{itemize}
\item[(A1)] \phantomsection\label{cnd:A1}
There exists $\bar{y}\in Y$ such that
$$\sup_{y\in Y}\int_0^{\infty}e^{-\lambda t}\int_{\Theta}\|{ w_{\theta}(S_i(t,\bar{y}))-\bar{y}}\|p(S_i(t,y),\theta)\,d\theta\,dt<\infty\;\;\; \text{for every}\;\;\; i\in I.$$
\item[(A2)] \phantomsection\label{cnd:A2}
There exist $\alpha\in(-\infty,\lambda)$, $L>0$ and some function 
\hbox{$\mathcal{L}:Y\to \mathbb{R}_+$}, bounded on bounded sets, such that
$$\|{S_i(t,y_1)-S_j(t,y_2)}\|\leq Le^{\alpha t}\|{y_1-y_2}\|+t\,\mathcal{L}({y_2})\,{d}(i,j)\;\;\; \text{for} \;\;\;t\geq 0,\; y_1,y_2\in Y,\; i,j\in I.$$
\item[(A3)] \phantomsection\label{cnd:A3}
There exists a~constant $L_w>0$ such that
$$\int_{\Theta} \|{w_{\theta}(y_1)-w_{\theta}(y_2)}\|p(y_1,\theta)\,d\theta\leq L_w\|{y_1-y_2}\|\;\;\;\text{for}\;\;\;y_1,y_2\in Y.$$%
\item[(A4)] \phantomsection\label{cnd:A4}
There exist $L_{\pi}>0$ and $L_p>0$ such that, for any $y_1,y_2\in Y$, $i\in I$,
$$\sum_{j\in I} |\pi_{ij}(y_1)-\pi_{ij}(y_2)|\leq L_{\pi} \|{y_1-y_2}\|
\;\;\;\text{and}\;\;\;
\int_{\Theta} |p(y_1,\theta)-p(y_2,\theta)|\,d\theta\leq L_{p} \|{y_1-y_2}\|.$$
\item[(A5)] \phantomsection\label{cnd:A5}
There exist $d_{\pi}>0$ and $d_{p}>0$ such that, for all $i_1,i_2\in I$, $y_1,y_2\in Y$,
$$\sum_{j\in I} \min\{\pi_{i_1,j}(y_1),\pi_{i_2,j}(y_2)\}\geq d_{\pi}
\;\;\;\text{and}\;\;\;
\int_{\Theta(y_1,y_2)}\min\{p(y_1,\theta),p(y_2,\theta)\}\,d\theta\geq d_p,$$
where 
$\Theta(y_1,y_2)=\{\theta\in\Theta:\, \|{w_{\theta}(y_1)-w_{\theta}(y_2)}\|\leq L_w \|{y_1-y_2}\|\}$.
\end{itemize}
In addition to this, we also assume that the constants appearing in conditions \hyperref[cnd:A2]{(A2)} and \hyperref[cnd:A3]{(A3)} satisfy the inequality
\begin{equation}\label{eq:balance1} LL_w+\alpha/\lambda<1. \end{equation}

We further investigate the sequence $(Y_n,\xi_n)_{n\in\mathbb{N}_0}$ with values in $X=Y\times I$. The space $X$ is assumed to be equipped with the metric given by
$$\varrho_{\tilde{c}}\left((y_1,i),(y_2,j)\right)=\|y_1-y_2\|+\tilde{c}\,{d}(i,j)\quad\text{for}\quad (y_1,i),(y_2,j)\in X,$$
where $\tilde{c}$ is a sufficiently large constant (defined explictly in \cite{dawid}), depending on $\lambda$, $\bar{y}$, $\alpha$, $\mathcal{L}$,  \hbox{$L$ and $L_w$.}

An easy computation shows that  $(Y_n,\xi_n)_{n\in\mathbb{N}_0}$ is a time-homogeneous Markov chain with transition law $\Pi_{\varepsilon}:X\times\mathcal{B}(X)\to\left[0,1\right]$ given by
\begin{align}\label{def:Pi_epsilon}
\Pi_{\varepsilon}(y,i,&A)=\int_0^{\infty}\lambda e^{-\lambda t}\int_{\Theta}p(S_i(t,y),\theta) \nonumber\\
&\times \int_{B(0,\varepsilon)}\left(\sum_{j\in I}\mathbbm{1}_A(w_{\theta}(S_i(t,y))+h,j)\,\pi_{ij}(w_{\theta}(S_i(t,y))+h)\right) \,\nu^{\varepsilon}(dh)\,d\theta\,dt\quad 
\end{align}
for any $(y,i)\in X$ and any $A\in\mathcal{B}_X$.

From the proof of \cite[Theorem 4.1]{dawid} it follows that, if conditions \hyperref[cnd:A1]{(A1)}-\hyperref[cnd:A5]{(A5)} hold with constants satisfying \eqref{eq:balance1}, then
the hypotheses of Theorem \ref{thm:spectral_gap} are fulfilled for $\Pi_{\varepsilon}$ and a~suitable substochastic kernel $Q$, satisfying \eqref{def:substoch}. Consequently, the Markov operator  corresponding to $\Pi_{\varepsilon}$ is then exponentially ergodic in the Fortet-Mourier metric induced by $\rho_{\tilde{c}}$ with a sufficiently large $\tilde{c}$.

\subsection{An Application of Theorem \ref{thm:CTG}}
In order to prove the CLT for the Markov chain $(Y_n,\xi_n)_{n\in\n_0}$, introduced in Section \ref{sec:model}, we strengthen assumptions \hyperref[cnd:A1]{(A1)} and \hyperref[cnd:A3]{(A3)} to the following conditions:
\begin{itemize}
\item[(A1)$^{\prime}$] \phantomsection\label{cnd:A1p}
 There exists $\bar{y}\in Y$ such that
$$\sup_{y\in Y}\int_0^{\infty}e^{-\lambda t}\int_{\Theta}\|{ w_{\theta}(S_i(t,\bar{y}))-\bar{y}}\|^2p(S_i(t,y),\theta)\,d\theta\,dt<\infty\quad \text{for}\quad i\in I.$$
\item[(A3)$^{\prime}$] \phantomsection\label{cnd:A3p}
There exists a~constant $L_w'>0$ such that
$$\int_{\Theta} \|{w_{\theta}(y_1)-w_{\theta}(y_2)}\|^2p(y_1,\theta)\,d\theta\leq L_w'\|{y_1-y_2}\|^2\;\;\;\text{for}\;\;\;y_1,y_2\in Y.$$
\end{itemize}
Let us note that, due to the H\"older inequality, conditions \hyperref[cnd:A1p]{(A1)$^{\prime}$}, \hyperref[cnd:A3p]{(A3)$^{\prime}$} imply \hyperref[cnd:A1]{(A1)}, \hyperref[cnd:A3]{(A3)}, respectively, and \hyperref[cnd:A3]{(A3)} holds with $L_w := \sqrt{L_w'}$. 

In the remainder of this section we assume that $V:X\to[0,\infty)$ is the Lyapunov function given by  
\begin{align}
\label{def:V_lil}
V(y,i)=\|y-\bar{y}\|\quad\text{for every}\quad(y,i)\in X,
\end{align}
where $\bar{y}$ is determined by \hyperref[cnd:A1p]{(A1)$^{\prime}$}.

\begin{theorem}\label{thm_CLT}
Consider the model stated in Section \ref{sec:model}. In particular, 
let $(Y_n,\xi_n)_{n\in\mathbb{N}_0}$ be the 
Markov chain with transition law $\Pi_{\varepsilon}$, given by \eqref{def:Pi_epsilon}, and initial distribution 
$\mu\in\mathcal{M}_1(X)$.  
Further, assume that conditions \hyperref[cnd:A1]{(A1)}-\hyperref[cnd:A5]{(A5)} with \hyperref[cnd:A1]{(A1)} and \hyperref[cnd:A3]{(A3)} strengthened to \hyperref[cnd:A1p]{(A1)$^{\prime}$} and \hyperref[cnd:A3p]{(A3)$^{\prime}$}, respectively, hold with 
\begin{align}
\label{CLT_condition}
L^2L_w^{\prime}+2\alpha\lambda^{-1}<1.
\end{align} 
Then, for every $g\in Lip_b(X)$, the chain $(g(Y_n,\xi_n))_{n\in\mathbb{N}_0}$ obeys the CLT, whenever its initial measure $\mu$ is such that $\mu\in\mathcal{M}_{1,1}^{V}(X)$ for $V$ given by \eqref{def:V_lil}. Moreover, if $(Y_n,\xi_n)_{n\in\mathbb{N}_0}$ is stationary, then $(g(Y_n,\xi_n))_{n\in\mathbb{N}_0}$ enjoys the Donsker invariance principle of the CLT.
\end{theorem}

\begin{proof}
We shall use Theorem \ref{thm:CTG}. It is easy to check that inequality \eqref{CLT_condition} implies \eqref{eq:balance1} with \hbox{$L_w := \sqrt{L_w'}$}. Hence, as mentioned in the previous section, conditions \hbox{\hyperref[cnd:B1]{(B1)}-\hyperref[cnd:B5]{(B5)}} for \hbox{$\Pi=\Pi_{\varepsilon}$} can be derived from \hyperref[cnd:A1]{(A1)}-\hyperref[cnd:A5]{(A5)} and \eqref{eq:balance1} (cf. the proof of \cite[Theorem~4.1]{dawid}). Obviously, \hyperref[cnd:B0]{(B0)} follows immediately from the continuity of functions $\pi_{i,j}$, $y\mapsto S_i(t,y)$, $y\mapsto p(y,\theta)$ and $w_{\theta}$.  
Thus the proof of Theorem \ref{thm_CLT} reduces now to showing condition \hyperref[cnd:B1p]{(B1)$^{\prime}$}, introduced in Section~\ref{subsec_CLT}.

Let $P_{\varepsilon}$ denote the Markov operator corresponding to $\Pi_{\varepsilon}$, and let $U_{\varepsilon}$ stand for its dual operator. Further, fix an arbitrary $(y,i)\in X$. We then see that
\begin{align}\label{eq:<V2,Pmu>}
\begin{aligned}
U_{\varepsilon} V^2(y,i) 
&=\int_X V^2(z,l)\,\Pi_{\varepsilon}\left(y,i,dz\times dl\right)\\
&
=\int_0^{\infty}\lambda e^{-\lambda t}\int_{\Theta}p(S_i(t,y),\theta)\int_{B(0,\varepsilon)}
\left\|w_{\theta}(S_i(t,y))+h-\bar{y}\right\|^2
\,\nu^{\varepsilon}(dh)\,d\theta \,dt.
\end{aligned}
\end{align}
Now, introduce $Z=[0,\infty)\times \Theta\times H$, and define $\nu\in \mathcal{M}_1(Z)$ as follows:
\begin{align*}
\nu(A)=&\int_0^{\infty}\lambda e^{-\lambda t}\int_{\Theta}p(S_i(t,y),\theta)\int_{B(0,\varepsilon)}\mathbbm{l}_A(t,\theta,h)\,\nu^{\varepsilon}(dh)\,d\theta\,dt,\;\;\; A\in\mathcal{B}_Z.
\end{align*}
Further, consider the space $\mathcal{L}^2(Z,\mathcal{B}_Z,\nu)$,  
and define $\varphi_0:Z\to\mathbb{R}$ by
\begin{align*}
\varphi_0(t,\theta,h)=\left\|w_{\theta}(S_i(t,y))+h-\bar{y}\right\|\quad\text{for}\quad (t,\theta,h)\in Z.
\end{align*}
Note that $\varphi_0\in \mathcal{L}^2(Z,\mathcal{B}_Z,\nu)$. To see this, let us first write
\begin{align*}
\varphi_0(t,\theta,h)\leq \left\|w_{\theta}(S_i(t,y))-w_{\theta}(S_i(t,\bar{y}))\right\|+\left\|w_{\theta}(S_i(t,\bar{y}))-\bar{y}\right\|+\|h\|.
\end{align*}
Then, using the Minkowski inequality, we obtain
\begin{align}\label{eq:CLT_model}
\begin{aligned}
\left(U_{\varepsilon} V^2(y,i) \right)^{1/2}
&= \left(\int_Z\varphi^2_0(t,\theta,h)\,\nu(dt\times d\theta\times dh)\right)^{1/2}\\
&\;\leq \left(\int_{Z}
\left\|w_{\theta}(S_i(t,y))-w_{\theta}(S_i(t,\bar{y}))\right\|^2\nu(dt\times d\theta\times dh)\right)^{1/2}\\
&\;+\left(\int_{Z}\left\|w_{\theta}(S_i(t,\bar{y}))-\bar{y}\right\|^2\nu(dt\times d\theta\times dh)\right)^{1/2}+\varepsilon,
\end{aligned}
\end{align}
where the second component on the right-hand side is finite due to assumption \hyperref[cnd:A1p]{(A1)$^{\prime}$}.
\hbox{According} to conditions \hyperref[cnd:A3p]{(A3)$^{\prime}$} and \hyperref[cnd:A2]{(A2)}, we further have 
\begin{align}
\begin{aligned}\label{eq:CLT_model_estim}
\int_{Z}&
\left\|w_{\theta}(S_i(t,y))-w_{\theta}(S_i(t,\bar{y}))\right\|^2\nu(dt\times d\theta\times dh)\\
&\leq
\int_0^{\infty}\lambda e^{-\lambda t} L_w'\|S_i(t,y)-S_i(t,\bar{y})\|^2\,dt\,
\leq 
\int_0^{\infty}\lambda e^{-\lambda t} L_w'L^2e^{2\alpha t}\|y-\bar{y}\|^2 \,dt\\
&=
\lambda L_w'L^2\|y-\bar{y}\|^2\left(\int_0^{\infty}e^{-(\lambda-2\alpha) t}\,dt\right) 
=
\frac{\lambda L_w^{\prime}L^2}{\lambda-2\alpha}V^2(y,i),
\end{aligned}
\end{align}
where the last equality follows from the fact that $2\alpha<\lambda$, which is ensured by  (\ref{CLT_condition}).  We see that, indeed, $\varphi_0\in \mathcal{L}^2(Z,\mathcal{B}_Z,\nu)$. 
Further, reffering to (\ref{eq:CLT_model}) and (\ref{eq:CLT_model_estim}), we obtain \hbox{condition~\hyperref[cnd:B1p]{(B1)$^{\prime}$}} with 
\begin{gather*}
a:=\sqrt{\frac{\lambda L_w^{\prime}L^2}{\lambda-2\alpha}},\\
b:=\sup_{y\in Y}\left(\int_0^{\infty}e^{-\lambda t}\int_{\Theta}\|{ w_{\theta}(S_i(t,\bar{y}))-\bar{y}}\|^2p(S_i(t,y),\theta)\,d\theta\,dt\right)^{1/2}+\varepsilon^*<\infty.
\end{gather*}
Moreover, due to assumption (\ref{CLT_condition}),  we see that $a\in(0,1)$, which completes the proof.
\end{proof}

\section*{Acknowledgements}
Hanna Wojew\'odka is supported by the Foundation for Polish Science (FNP). Part of this work was done when Hanna Wojew\'odka attended a four-week study trip to the Mathematical Institute at Leiden University, which was also supported by the FNP (the so-called "Outgoing Stipend" in the START programme).

\bibliography{references}
\bibliographystyle{plain}
\end{document}